\documentclass[12pt]{article}
\usepackage{amsmath,amsthm,amsfonts,amssymb,verbatim}
\usepackage[usenames]{color} 
\newtheorem{thm}{Theorem}[section]
\newtheorem{lemma}[thm]{Lemma}

\theoremstyle{definition}

\newtheorem{remark}[thm]{Remark}

\theoremstyle{remark}
\DeclareMathOperator{\Fix}{Fix}

\DeclareMathOperator{\Diff}{Diff}

\DeclareMathOperator{\mcg}{MCG}

\DeclareMathOperator{\Homeo}{Homeo}
\DeclareMathOperator{\Int}{int}

\DeclareMathOperator{\Cent}{Cent}
\DeclareMathOperator{\genus}{genus}

\newcommand{\R}{\mathbb R}

\newcommand{\Z}{\mathbb Z}

\def\ti{\tilde}
\def\sinfty{S_{\infty}}
\def\sl3z{SL(3, \mathbb Z)}

\def\G{{\cal G}}
\def\H{{\cal H}}
\def\L{{\cal L}}

\title{Global fixed points for centralizers and Morita's Theorem}
\author{John Franks\thanks{Supported in part by NSF grant DMS-055463.} ,\ \ 
Michael Handel\thanks{Supported in part by NSF grant DMS-0103435.}}

\begin{document}
\maketitle
\begin{abstract}
We prove a global fixed point theorem for the centralizer of a
homeomorphism of the two dimensional disk $D$ that has
attractor-repeller dynamics on the boundary with at least two
attractors and two repellers.  As one application, we show that there
is a finite index subgroup of the centralizer of a pseudo-Anosov
homeomorphism with infinitely many global fixed points.  As another
application we give an elementary proof of Morita's Theorem, that the
mapping class group of a closed surface $S$ of genus $g$ does not lift
to the group of diffeormorphisms of $S$ and we improve the lower bound
for $g$ from $5$ to $3$.  \end{abstract}

\section{Introduction}

In this article we are concerned with the properties of
groups of homeomorphisms or diffeomorphisms of surfaces.  We assume
throughout that $S$ is a surface of finite negative Euler
characteristic, without boundary but perhaps with punctures.  We
denote the group of orientation preserving homeomorphisms of $S$ and
the group of orientation preserving $C^1$ diffeomorphisms of $S$ by
$\Homeo(S)$ and $\Diff(S)$ respectively and we denote the subgroups
consisting of elements that are isotopic to the identity by
$\Homeo_0(S)$ and $\Diff_0(S)$ respectively.

An important tool in the study of subgroups of these groups is the
existence of a global fixed point.  A {\em global fixed point} for a
subgroup $\G$ of $\Homeo(S)$ is a point $x \in S$ that is fixed by
each element of $\G$.  The set of global fixed points for $\G$ is
denoted $\Fix(\G)$.  When $G$ is a subgroup of $\Diff(S)$ and
$\Fix(\G)$ is non-empty, the assignment $g \mapsto Dg_x$ (the
derivative of $g$ at $x \in \Fix(\G)$), gives a representation of $\G$
in $Gl(2,\R)$.  This representation can be very useful for
understanding $\G$;  for example in \cite{FH} this representation was
used to prove that many lattices, including $SL(3,\Z),$ are not isomorphic
to a subgroup of the group of measure preserving
diffeomorphisms of a surface.

Unfortunately, there are no general techniques for finding a global
fixed point for a subgroup of $\Homeo(S)$.  In particular there are
no analogues of the standard tools of algebraic topology for finding
fixed points of a single map.  In the case of surfaces there is a 
literature concerning the existence of global fixed points,
but it is largely limited to abelian groups (see, e.g.
\cite{B1}, \cite{B2}, \cite{H}, \cite{FHP}, and \cite{FHP2}).

The main objective of this article is to provide a technique which,
in many cases, allows us to find a global fixed point for the centralizer
of $f \in \Homeo(S)$.  We denote the centralizer of $f$ by $\Cent(f)$ and
observe that it can be a very large group and, in particular, is generally far
from abelian.

As one application of this result we address the ``lifting problem''
for the mapping class group (see \S 6 of \cite{farb}).  Using
our result on global fixed points and the representation in $Gl(2,\R)$
mentioned above, we give an elementary proof of an important theorem of Morita
about the non-existence of liftings of the full mapping class group to
$\Diff(S)$ and improve the lower bound on the genus of $S$ required
for the result (see Theorem~\ref{morita} below).

The closed two dimensional disk is denoted $D$.  The universal cover
$\ti S$ of $S$ is naturally identified with $\Int D$ and the
compactification of $\ti S$ by the circle at infinity $\sinfty$ is
naturally identified with $D$.  For this reason, our main result
concerns global fixed points for group actions on $D$.
  
\begin{thm}\label{thm: global fixed point}
Let $\G$ be a subgroup of $\Homeo(D)$ and let $f$ be an
element of the center of $\G.$
Suppose $K:= \Fix(f) \cap \partial D$ consists of a finite set 
with more than two elements each of which is either an
attracting or repelling fixed point for $f: D \to D.$
Let $\G_0 \subset \G$ denote the finite index subgroup
whose elements pointwise fix $K$.
Then $\Fix(\G_0) \cap \Int(D)$ is non-empty.
\end{thm}

The hypothesis of this theorem has both an algebraic part, namely that
$\G_0 \subset \Cent(f)$, and a dynamical part, namely that
$f|_{\partial D} $ has attractor-repeller dynamics.  The latter
implies that $\Fix(f) \cap \Int D \ne \emptyset$.  The former is used
to relate the dynamics of elements of $\G_0$ to $f$ and hence to each
other.

The fixed point set $\Fix(f)$ of $f \in \Homeo(S)$ is partitioned into
{\em Nielsen classes}.  Two elements $x,y \in \Fix(f)$ belong to the
same Nielsen class if there is a lift $\ti f:\ti S \to \ti S$ of $f :
S \to S$ and lifts $\ti x, \ti y \in \Fix(\ti f)$ of $x$ and $y$.
Equivalently, every lift of $f$ that fixes a lift of $x$ also fixes a
lift of $y$.  If $f$ is isotopic to $g$ and $\ti f$ is a lift of $f$
then the isotopy between $f$ and $g$ lifts to an isotopy between $\ti
f$ and a lift $\ti g$ of $g$.  We say that $\ti f$ and $\ti g$ are
{\em paired}.  Pairing defines a bijection between lifts of $f$ and
lifts of $g$ and we say that the $f$-Nielsen class of $x \in \Fix(f)$
{\em is paired} with the $g$-Nielsen class of $z \in \Fix(g)$ if there
are paired lifts $\ti f$ and $\ti g$ of $f$ and $g$ and there are
lifts $\ti x \in \Fix(\ti f)$ of $x$ and $\ti z \in \Fix(\ti g)$ of
$z$.

  By a pseudo-Anosov homeomorphism of a punctured surface we mean the restriction of a homeomorphism of the unpunctured surface that is pseudo-Anosov relative to the set of  punctures.

\begin{thm} \label{cor:pA centralizer} 
Suppose that $\alpha \in \Homeo(S)$ is pseudo-Anosov and that $f\in
\Homeo(S)$ is isotopic to $\alpha$. Let $\H_0 = \Cent (f) \cap
\Homeo_0(S)$.  Then $\Fix(\H_0)$ is infinite.  More precisely, for
each $n \ge 0$ and each $y \in \Fix(\alpha^n)$ there exists
$x\in\Fix(\H_0) \cap \Fix(f^n)$ such that the Nielsen
class of $f^n$ determined by $x$ is paired with the Nielsen class of
$\alpha^n$ determined by $y$.
\end{thm}

In the special case that $f = \alpha$, the group $\H_0$ is trivial (see,
e.g. \cite{FLP}) and Theorem~\ref{cor:pA centralizer} is an immediate
consequence of the fact that $\alpha$ has the fewest possible periodic
points in its isotopy class.  Thus Theorem~\ref{cor:pA centralizer}
fits into the general scheme of results in surface dynamics in which
an important property of pseudo-Anosov maps is extended to all
elements of its isotopy class.
 
\begin{remark} 
Suppose that $f' :T^2 \to T^2$ is isotopic to a linear Anosov
homeomorphism $\alpha' :T^2 \to T^2$ and that $e \in\Fix(\alpha)$ is
the image of $(0,0) \in \R^2$ under the usual covering map.  Let $S =
T^2 \setminus \{e\}$, let $f=f'|_{S}$ and let $\alpha = \alpha'|_{S}$.
Then $\alpha$ is pseudo-Anosov and we may apply Theorem~\ref{cor:pA
centralizer}.  The conclusions are exactly as given in the theorem but
are applied to the subgroup $\H'_0$ of $\Cent(f')$ that are isotopic
the identity relative to $e$.
 \end{remark}

Our next result is the analogue of Theorem~\ref{cor:pA centralizer}
for reducible isotopy classes with a pseudo-Anosov component.  It
is a corollary of, and the original motivation for, Theorem~\ref{thm:
global fixed point}, since it provides the tool we use to prove
Morita's theorem.

\begin{thm}  \label{pAcomponent}Suppose that $f\in \Homeo(S)$, that $S_0 \subset S$ is an incompressible subsurface and that $f$ is isotopic to $\alpha \in \Homeo(S)$ where $\alpha(S_0) = S_0$ and   $\alpha|_{S_0}$ is pseudo-Anosov.    Let $\H_0$ be the subgroup of $\Cent (f)$ consisting of elements that are isotopic to a homeomorphism  that pointwise fixes  $S_0$.        Then $\Fix(\H_0)$ is infinite.  More precisely,  for each $n \ge 0$ and each   $y \in \Fix(\alpha^n) \cap  S_0$ there  exists $x\in\Fix(\H_0) \cap \Fix(f^n)\cap S$  such that the Nielsen class of $f^n$ determined by $x$ is paired with the Nielsen class of $\alpha^n$ determined by $y$.
\end{thm}

 The {\em mapping class group} $\mcg(S)$ of a closed surface $S$ is
the group of isotopy classes of orientation preserving homeomorphisms
of $S$. There is a natural homomorphism $\Homeo(S) \to \mcg(S)$ that
sends $h \in \Homeo(S)$ to its isotopy class $[h] \in \mcg(S)$.  A
{\em lift }of a subgroup $\Gamma$ of $\mcg(S)$ is a homomorphism
$\Gamma \to \Homeo(S)$ so that the composition
\[
 \Gamma \to \Homeo(S)\to \mcg(S)
\]
is the inclusion.  Every free abelian subgroup (see \S 6.3 of
\cite{farb}) and every finite subgroup \cite{ker} of $\mcg(S)$ has a
lift to $\Diff(S)$.  Morita \cite{mor1}, \cite{mor2} proved that
$\mcg(S)$ does not lift to $\Diff(S)$ for $\genus(S) \ge 5$ and
Markovich \cite{mark} proved that $\mcg(S)$ does not lift to
$\Homeo(S)$ for $\genus(S)\ge 6$.

Using Theorem~\ref{pAcomponent} and the Thurston stability theorem we
give an elementary proof of Morita's theorem and improve the lower
bound on the genus.

 \begin{thm} \label{morita}  $\mcg(S)$ does not lift to $\Diff(S)$ for $\genus(S) \ge 3$.
\end{thm}      
 
This is actually a special case of a more general result in which we
consider homomorphisms $\L : \Gamma \to \Diff(S)$ that are not
necessarily lifts.
   
Suppose that $S =S_1 \cup S_2$ where $S_1$ and $S_2$ are
incompressible subsurfaces with disjoint interiors.  Recall that the
relative mapping class group $\mcg(S_i,\partial S_i)$ is the set of
isotopy classes rel $\partial S_i$ of homeomorphisms $h:S_i\to S_i$
that are the identity on $\partial S_i$.  Each such $h$ extends by the
identity to a homeomorphism of $S$.  This induces a homomorphism
$\Phi: \mcg(S_i,\partial S_i) \to \mcg(S).$ If $S$ is given a
hyperbolic structure with $\partial S_i$ a union of geodesics, it is
straightforward to see that any $f \in \Homeo_0(S)$ is actually
homotopic to the identity along geodesics, i.e. with $f_t(x)$ on the
geodesic from $x$ to $f(x)$ determined by the isotopy. From this it
follows that if $f$ is the identity on the complement of $S_i$, then
we may choose a homotopy from $f$ to the identity with the same
property, and hence an isotopy of $f|_{S_i}$ to the identity rel
$\partial S_i$.  Therefore the homomorphism $\Phi$ is injective and we
use it to identify $\mcg(S_i,\partial S_i)$ with a subgroup of
$\mcg(S)$.

\begin{thm}  \label{thm:Gamma} Assume notation as above and that  $\Gamma = \langle\Gamma _1, \mu \rangle$ where
\begin{itemize}
\item $\Gamma_1$ is  a non-trivial finitely generated subgroup of $\mcg(S_1,\partial S_1)$ such that 
$H^1(\Gamma_1)$ has rank zero. 
\item   $\mu \in \mcg(S_2,\partial S_2)$.
\end{itemize}
Then there does not exist a faithful homomorphism $\L : \Gamma \to \Diff(S)$ such that  
\begin{enumerate}
\item $ [\L(\Gamma_1)] \subset \mcg(S_1,\partial S_1))$.
\item  $[\L(\mu )]     \in \mcg(S_2,\partial S)$ is represented by  $\alpha:S \to S$ where $\alpha(S_2)=S_2$  and $\alpha|S_2$ is   pseudo-Anosov.
 \end{enumerate}
\end{thm}

Theorem~\ref{morita} follows from Theorem~\ref{thm:Gamma} and the fact \cite{kork} that $H^1(\mcg(S_1,\partial S_1)$ has rank zero.    It is conjectured that $H(\Gamma)$ has rank zero for all finite index subgroups $\Gamma$ of $\mcg(S)$.    If that conjecture is verified then Theorem~\ref{thm:Gamma} will imply that no finite index subgroup of $\mcg(S)$ lifts to $\Diff(S)$, which is known  for genus at least $5$ because Morita's original proof applies to all finite index subgroups of $\mcg(S)$.  

We thank Benson Farb for several very helpful conversations.

\section{A Global Fixed Point Theorem}  \label{sec:gfp}

We begin with generalities about attracting fixed points.

Suppose that $f : X \to X$  is a homeomorphism of a locally compact metric space $X$ and that $x_0 \in \Fix(h)$.  We say that $x_0$ is an {\em attracting fixed point} for $f$    if there is a  compact neighborhood $W$ of $x_0$ such that  the $f^n(W) \to \{x_0\}$ in the Hausdorff topology; i.e. for every neighborhood $N$ of $x_0$, we have $f^n(W) \subset N$ for all sufficiently large $n$.   
The {\em basin of attraction of $x_0$ with respect to $f$} is defined
to be $\{x \in X:\lim_{n\to \infty}f^n(x)= x_0\}$.   Note that the basin
of attraction of $x_0$ is $f$-invariant and contains $W$. 

\begin{remark}  \label{intersection is a point} If $W$ is a compact neighborhood of $x_0$ such that $f(W) \subset W$ then $f^n(W) \to \{x_0\}$ in the Hausdorff topology if and only if    $ \cap_{n=1}^\infty f^n(W) = \{x_0\}$.  Thus by  item (3) of Lemma~\ref{attracting nbhds} below, $x_0$ is an attracting fixed point if and only if  it has a compact neighborhood $W$ such that  $f(W) \subset W$ and $ \cap_{n=1}^\infty f^n(W) = \{x_0\}$.  
\end{remark}

    If $x$ is
an attracting fixed point for $f^{-1}$ then it is a {\em repelling}
fixed point for $f$.

\begin{lemma}\label{attracting nbhds} Let $f: X \to X$ be a homeomorphism of a locally compact
metric space with an attracting fixed point $x_0 \in X$ and basin of attraction $U$.    
\begin{enumerate}
\item  For any   compact neighborhood $W_0 \subset U$ of $x_0$, 
\[
\bigcup_{i=0}^\infty f^{-i}(W_0) = \bigcup_{i=0}^\infty f^{-i}(\Int W_0) = U.
\]
\item For any compact set $A \subset U,\ 
\lim_{n \to \infty}f^n(A) = \{x_0\}$ in the Hausdorff topology.
\item There exists arbitrarily small compact neighborhoods $V $ of $x$
such that $f(V) \subset V$.  
\item If $x_0$ is also an attracting
fixed point for a homeomorphism $g:X \to X$ that commutes with $f$
then $U$ is the basin of attraction of $x$ with respect to $g$.
\end{enumerate}
\end{lemma}

\proof Let $W$ be a compact neighborhood of $x_0$ as in the definition
of attracting fixed point.

Suppose that $W_0 \subset U$ is a compact neighborhood of
$x_0$ .  Since $U$ is $f^{-1}$-invariant, $\cup_{i=0}^\infty f^{-i}(W_0) 
\subset U$.  Conversely, if $x \in U$ then $f^i(x) \in \Int W_0$
for all sufficiently large $i$.  This proves that $ U \subset
\cup_{i=0}^\infty f^{-i}(\Int W_0) \subset \cup_{i=0}^\infty f^{-i}(W_0)
\subset U$ which proves   (1).

For (2), suppose  that $A$ is a compact subset of 
$U $ and that   $N$  is a compact neighborhood of $x$.  Then $U=   \cup_{i=0}^\infty f^{-i}(\Int N)$ by (1) and hence  $A \subset f^{-m}(\Int N)$ for some $m >0$.  Thus $f^{ i}(A) \subset  (N)$ for all $i \ge m$ and
we conclude $\lim_{n \to \infty}f^n(A) \subset \cap_{i=0}^\infty f^{i}(W)$.  This proves (2).

For (3), choose $q > 0$
such that $f^q(W) \subset \Int(W)$  and define
$V_1 = \bigcup_{k=0}^q f^k(W)\subset U.$
 Then   $f(V_1) \subset V_1$.  Given a  neighborhood  $N$  of $x$ choose $l \ge 0$ so that $f^l(V_1) \subset N$  and let     $V = f^l(V_1)$.  Then  $V \subset N$  and $f(V) \subset V$.
 
Suppose now that $g$ is as in (4) and that $U'$ is the basin of
attraction of $x_0$ with respect to $g$.  For any compact neighborhood
$N \subset U \cap U'$ of $x_0$,\ $U = \cup_{i=0}^\infty f^{-i}(\Int N)$
and $U' = \cup_{j=0}^\infty g^{-j}(\Int N)$ by (1).  For all $i\ge 0$
there exists $j\ge 0$ so that $g^j(N) \subset f^i(\Int N)$.  Thus
$g^j(f^{-i}(N)) = f^{-i}(g^j(N)) \subset \Int N$ or equivalently
$f^{-i}( N) \subset g^{-j}(\Int N)$. This proves that $U \subset U'$.
The reverse inclusion follows by symmetry.  \endproof

\begin{lemma} \label{lem: attracting2} Let $f: X \to X$ be a homeomorphism of a locally compact
metric space and let $x_0 \in X$ be   an attracting fixed point for  $f$. If $g: X \to X$ is a homeomorphism that commutes with $f$ and fixes $x_0$ then there exists $m >0$ such that $x_0$ is an attracting fixed point for $h = f^m g$.    
\end{lemma}
\begin{proof}
Let $U$ be the basin of attraction for $x_0$ with respect to $f$.  Then $g(U)$ is the basin of attraction for $x_0$ with respect to $gfg^{-1} = f$ which implies that $g(U) = U$.     
By part (3) of Lemma~\ref{attracting nbhds} there is a compact
neighborhood $V$ of $x_0$ such that $f(V) \subset V$.  
By part (2) of the same lemma  there is $m>0$ such that $f^m(g(V))
\subset f(V)$. 

Define $h = f^m g$.  Thus $h(V) \subset f(V) \subset
V.$ Applying $h^{n-1}$ we conclude $h^n(V) \subset f h^{n-1}(V)$ for
all $n$.  Hence
\[
f^{-1}h^n(V) \subset h^{n-1}(V)
\]
for all $n$.  

Let $\Lambda$ be the non-empty compact set $\cap_{n \ge 0} h^n(V).$
Then $f(\Lambda) \subset \Lambda$ because $f(V) \subset V$.  Also, the
displayed inclusion above implies $f^{-1}(\Lambda) \subset \Lambda$.
We conclude that $f(\Lambda) = \Lambda$ so by part (2) of
Lemma~\ref{attracting nbhds} the only possibility is that $\Lambda =
\{x_0\}$.  By Remark~\ref{intersection is a point}, this proves that $x_0$ is an attracting fixed point for
$h$.  
\end{proof}

We now turn to the proof of our main result.
\vspace{.1in}


\begin{proof}[\bf Proof of Theorem~\ref{thm: global fixed point}]
The points of $K$ are attractors or repellers for $f$ and hence {\it
a fortiori} attractors or repellers for $f|_{\partial D}$.  Hence
there must be an equal number of attractors and repellers which
alternate on $\partial D.$ We will denote the attractors $\{p_1,\dots,
p_k\}$ and the repellers $\{q_1,\dots, q_k\}$ in their circular order.

For simplicity it is useful to consider the two sphere $S^2$ obtained
by doubling $D$ along its boundary.  There is a natural extension
of $f$ to $S^2$ which we will also denote by $f$.  Then 
$\{p_1,\dots, p_k\}$ are attracting fixed points of $f:S^2 \to S^2$
and $\{q_1,\dots, q_k\}$ are repelling fixed points.  Each element
of $\G_0$ also extends in a natural way to $S^2$ and abusing notation
we will denote this group by $\G_0$. 

Let $\H$ be the set of elements of $\langle \G_0,f \rangle$ for which
each $p_i$ is an attractor and each $q_i$ is a repeller.  Then $f \in
\H$ and by Lemma~\ref{lem: attracting2}, for all $g \in \G_0$ there
exists $m > 0$ such that $f^mg \in \H$.  Our goal is to find $y \in
S^2$, not in the set $\{p_1,\dots, p_k\} \cup \{q_1,\dots, q_k\}$,
such that $y \in \Fix(h)$ for each $h \in \H$.  We then observe that
$y \in \Fix(f) \cap \Fix(f^mg) \subset \Fix(g)$ for all $g \in \G_0$
which will complete the proof.
   
\begin{remark}   
An easy index argument shows that for each $h \in \H$ there is at
least one element of $\Fix(h)$ that has negative index and so is
neither a source nor a sink.  The challenge here is to find a single
point that works for all $h$.  The point we find in $\Fix(\H)$ will be
shown to be neither a source nor a sink for any element $h \in \H$,
but we do not know about its index.
\end{remark}

Lemma~\ref{attracting nbhds}-(4) and the fact that each $h \in \H$
commutes with $f$ imply that the basin of attraction $U$ for $p_1$
with respect to $h \in \H$ is independent of $h$.  By
Lemma~\ref{attracting nbhds}-(1), $U$ can be written as an increasing
union of open disks and so is connected and simply connected.  The
fact that there is another attractor $p_2$ implies that the frontier
of $U$ is not a single point.

 We will be interested in two compactifications of $U$.  The first is
$\bar U$, the closure of $U$ in $S^2$ and the second is $\hat U$, the
{\em prime end compactification}.  The set $\Gamma = \widehat U \setminus U$ of
{\em prime ends} is topologically a circle. Each homeomorphism $h|_U:U \to U$
extends to homeomorphism $\hat h: \widehat U \to \widehat U$. Moreover,
 \begin{itemize}
 \item[$\ast$)] For each continuous arc $\gamma: [0,1] \to  \bar U$ with
$\gamma([0,1)) \subset U$ and $\gamma(1)$ in the frontier of $U$   there is a continuous arc $\widehat \gamma: [0,1] \to \widehat U$ with $\widehat
\gamma(t) = \gamma(t)$ for $t \in [0,1).$ The point $\gamma(1)$ is
called an {\em accessible point} of the frontier of $U$ and $\widehat \gamma(1)$
is a prime end corresponding to it (there may be more than one prime
end corresponding to an accessible point).
\end{itemize}
These properties  go back to Caratheodory.  An excellent
modern exposition can be found in Mather's paper \cite{M}. In
particular see \S 17 of \cite{M} for a discussion of accessible points. 
 
   Let $C$ denote the arc in $\partial D$ joining $p_1$ and $q_1$ and
not containing any other points of $\Fix(h).$  Then $C \subset \bar U$
and all but the endpoint $q_1$ of $C$ lies in $U$.  By ($\ast$), the half open arc 
$C \cap U$ converges to a   prime end
$\hat w \in \Gamma$  that is fixed by each $\hat h$ since $C$ is  $h$-invariant for each $h$.

Our first objective is to show that $\hat w$ is a repelling fixed
point for  each $\hat h|_\Gamma$ or equivalently an
attracting fixed point for each $\hat h^{-1}|_\Gamma.$ For this, there is no  loss in  replacing $h$ by an iterate so  by applying Lemma~\ref{attracting nbhds}-(2) we may choose  a disk neighborhood $D_1$ of $q_1$   in the  basin of attraction of $q_1$ with respect to $h^{-1}$ such that $h^{-1}(D_1) \subset \Int(D_1)$.
Let $\gamma_0$ denote the component of  $\partial D_1 \cap U$ which intersects $C$,  and let   $\gamma_i =
h^{-i}(\gamma_0).$   Note that each $\gamma_i$ intersects $C$ and is  the interior of a closed path $\bar \gamma_i \subset \bar U$ and so by $(\ast)$ is the  interior of a closed path $\widehat \gamma_i \subset \hat U$.  The  $\gamma_i$'s separate $U$ into two complementary components $V_i$ and $W_i$ with $W_{i+1}  \subset W_i$ and  $\cap_{i=1}^\infty W_i = \emptyset$ (where the last property follows from   Remark~\ref{intersection is a point}) .    
By Corollary 4 of \cite{M}, the $\hat \gamma_i$'s converge to    a single prime end, which must be $\hat w$ since each $\hat \gamma_i$ intersects $C$.  The $\hat h^{-1}$-orbit of the endpoints of $\hat \gamma_i$ converge to $\hat w$.  This proves that $\hat w$ is an  attractor for $\hat h^{-1}|_\Gamma$.

The next step is to find a prime end $\hat y$ that is fixed by each
$\hat h$ and that does not come from $\partial D$.  Since each $\hat h$
commutes with $\hat f$, Lemma~\ref{attracting nbhds}-(4) implies that the
interval of attraction of $\hat w$ with respect to $\hat h^{-1}$ is
independent of $h$.  Let $\hat y$ be one endpoint of this interval of
attraction and let $J$ be the interval in $\Gamma$ with endpoints
$\hat w$ and $\hat y$ which lies in the basin of $\hat w$.  The only
fixed points of $\hat h|_J$ are the endpoints and $\hat w$ is a
repeller while $\hat y$ is an attractor.
  
  Now we show how to extract a (not necessarily unique)  point $y \in \bar U$ from $\hat y$ that is fixed by each $h$.   By Corollary~11 and Theorem~13 of \cite{M}  there is a sequence of disjoint closed arcs $\bar \alpha_i \subset \bar U$  with interior  $\alpha_i \subset U$ and endpoints in the frontier of $U$ (in $S^2$)  and there exists  $y$ in the frontier of $U$ such that 
  \begin{itemize}
  \item   $ \bar \alpha_i \to y $    (in the Hausdorff topology on $S^2$) 
  \item there is a component $Z_i$ of  $(U \setminus \alpha_i)$ such that $Z_{i+1} \subset Z_i$ and $\cap_{i=1}^{\infty}Z_i = \emptyset$
  \item     $\alpha$ extends to a closed arc $\hat \alpha_i$ in $\hat U$ such that $\hat \alpha_i$ converges to $\hat y$ ( in the Hausdorff topology on $\hat U$)
  \end{itemize}

Clearly if we show $h(\alpha_i) \cap \alpha_i \ne \emptyset$ then it
follows that $y \in \Fix(h).$ One of the endpoints of $\hat \alpha_i$,
call it $\hat x_i$, lies in $J$.  Since $\hat h$ has no fixed points
in the interior of $J$ it follows that
\[
\lim_{n\to \infty} \hat h^n(\hat x_i) = \hat y \text{ and }
\lim_{n\to -\infty} \hat h^n(\hat x_i) = \hat w.
\]
On the other hand each $z \in \alpha_i$ lies in 
$U$, the basin of attraction of $p_1,$ so
\[
\lim_{n\to \infty}h^n(z) = p_1.
\]
Since $\hat \alpha_i$ separates $\hat y$ and $p_1$ in $\hat U$ we
conclude that $h^n(\alpha_i) \cap \alpha_i \ne \emptyset$ for large
$n$ (and hence also $\alpha_i \cap h^{-n}(\alpha_i) \ne \emptyset$).
But this implies $h^k(\alpha_i) \cap \alpha_i \ne \emptyset$ for all
$k \in \Z$ because $h^k(\alpha_i) \cap \alpha_i = \emptyset$ and $k
\ne 0$ would imply that for all $n>0$ either $h^{nk}(\alpha_i) \subset
Z_i$ or $h^{-nk}(\alpha_i) \subset Z_i$ which is a contradiction.
Hence the point $y$ is fixed by each $h$.

Moreover, for any $h \in \H$ we have $h^k(\alpha_i) \cap \alpha_i \ne
\emptyset$ for all $k \in \Z$.  This implies $y$ is not an attractor
for either $h$ or $h^{-1}$ since otherwise, choosing $A = \alpha_i$ for some
large $i$, we would contradict Lemma~\ref{attracting nbhds}-(2). 
We conclude $y$ is not contained in $ \{p_1,\dots, p_k\}
\cup \{q_1,\dots, q_k\}$.
\end{proof}

\section{Applications}  \label{sec:morita}

\noindent{\bf Proof of Theorem~\ref{cor:pA centralizer} and of Theorem~\ref{pAcomponent} }   We assume the notation of Theorem~\ref{pAcomponent} and allow the possibility that $S_0 =  S$.
 
We use the standard setup for discussing Nielsen classes in surfaces.
The universal cover $\ti S$ of $S$ is topologically the interior of a
disk and is compactified to a closed disk $D$ by the \lq circle at
infinity\rq\ $S_{\infty}$.  Every lift $\ti g :\ti S\to \ti S $ of
every homeomorphism $g:S \to S$ extends to a homeomorphism $\hat g : D
\to D$.  If an isotopy between $g_1$ and $g_2$ is lifted to an isotopy
between lifts $\ti g_1$ and $\ti g_2$ then $\hat g_1|_{S_{\infty}} =
\hat g_2|_{S_{\infty}}$.

Choose a component $\ti S_0$ of the full pre-image of $S_0$ in   $\ti S$ and let $C$ be the intersection of the closure of $\ti S_0$ in $D$ with $S_{\infty}$.   Then $C$ is a Cantor set if $S \ne S_0$ and   $C = S_{\infty}$ if $S = S_0$.    Since $C$ contains at least three points there is at most one lift of any homeomorphism that pointwise fixes $C$.  

Given $h \in \H_0$ choose  $g:S \to S$ that is isotopic to $h$ and pointwise fixes $S_0$ and let $\ti g: \ti S \to \ti S$ be the lift of $g$ that pointwise fixes $\ti S_0$.    The isotopy from $h$ to $g$ lifts to an isotopy from $\ti h$ to a lift $\ti g$  of $g$ satisfying $\hat h|_{C} = \hat g|_{C} = $identity.  The assignment $h \to \hat  h$ defines a lift $\widehat{\H_0} \subset \Homeo(D)$ of $\H_0$.  

Given $y \in \Fix(\alpha^n) \cap S_0$, choose a lift $\ti y \in \ti
S_0$ of $y$, let $A : \ti S \to \ti S$ be the lift of $\alpha^n$ that
fixes $\ti y$ and note that $A$ preserves $\ti S_0$.  Then $C$ is
$\hat A $-invariant and $\Fix(\hat A|_{S_{\infty}})$ is a finite
subset of $C$ with more than two elements. Moreover, each point of
$\Fix(\hat A|_{S_{\infty}})$ is either an attracting or repelling
fixed point for $\hat A: D \to D$ (see, for example, Theorem 5.5 of
\cite{CB} or \cite{HT}).  The isotopy from $\alpha^n$ to $f^n$ lifts
to an isotopy from $A $ to a lift $F$ of $f^n$ such that $\hat F =
\hat A$.  The commutator $[\hat F, \hat h]$ is the identity because it
fixes each point in $C$ and is the extension of a lift of $[f^n,h]=$
identity.  This proves that $\hat F$ commutes with each $\hat h$.

    Theorem~\ref{thm: global fixed point} applies to $\G = \G_0 = \langle \widehat{\H_0}, F\rangle$.  We conclude that there exists $\ti x \in \Fix(\widehat{\H_0}) \cap \Fix(F) \cap \ti S$.   The image $x \in S$ of $\ti x$ satisfies the conclusions of the theorem.
\qed

\vspace{.1in}
\noindent{\bf Proof of Theorem~\ref{thm:Gamma}}    We assume that   there  is a faithful homomorphism $\L : \mcg(S) \to \Diff(S)$ satisfying (1) and (2) and    prove that there is a non-trivial homomorphism from $\L(\Gamma_1)$ to $\R$, thereby contradicting the assumption that $H^1(\Gamma_1 )$ has rank zero.    

    By hypothesis,  each $h \in \L(\Gamma_1)$ is  isotopic  to a homeomorphism that is the identity on $S_2$ and $f := \L(\mu)$ is isotopic to a homeomorphism $\alpha$ such that $\alpha(S_2) = S_2$ and $\alpha|_{S_2}$ is pseudo-Anosov.   Also $f$ commutes with each $h \in \L(\Gamma_1)$ because $\mu$ commutes with each element of $\Gamma_1$.    Theorem~\ref{pAcomponent} implies that   $\Fix( \L(\Gamma_1))$ is infinite. 
    
     Choose a  non-isolated point $x$ of $\Fix(\L(\Gamma_1))$.     
  The assignment     $h \to \det(Dh_x)$ defines a homomorphism from  $\L(\Gamma_1)$ to $\R$.  If this is non-trivial we are done.  Otherwise, each   $Dh_x$ has determinant one.   Since $x$ is the limit of global fixed points there is a vector based at $x$ that is fixed by each $Dh_x$.  Thus there is a basis for the tangent space of $S$ at $x$ with respect to which each $Dh_x = \left(
\begin{array}{cc}
1 & n_h\\
0 & 1 \\
\end{array}
\right)$.   The map $h  \to n_h$ defines a homomorphism from $\L(\Gamma_1)$ to $\Z$.  If it is non-trivial we are done. Otherwise $n_h = 0$  and   $Dh_x$ is the identity for all $h\in \L(\Gamma_1)$.   The  existence of a  non-trivial homomorphism from $\L(\Gamma_1)$ to $\R$ now follows from the Thurston stability theorem (\cite{thurs}, see also Theorem~3.4 of \cite{fr:survey}).
\qed

\end{document}